\pdfoutput=1
\documentclass[11 pt, oneside, reqno]{amsart}
\usepackage[top=1.5in,bottom=1in,left=1in,right=1in]{geometry}
\usepackage{amsmath,amssymb,amsthm,amsopn,extarrows,accents,faktor,enumitem,stmaryrd,mathtools}  
\usepackage[linktocpage=true]{hyperref}
\hypersetup{
    colorlinks = true,
    allcolors = {blue}}
\usepackage[dvipsnames]{xcolor}

\usepackage{setspace}
\usepackage{tikz-cd,tikz}
\usetikzlibrary{matrix,arrows,decorations.pathmorphing}

\newtheorem{lemma}[equation]{Lemma}
\newtheorem{theorem}[equation]{Theorem}
\newtheorem*{theorem*}{Theorem}
\newtheorem{proposition}[equation]{Proposition}

\theoremstyle{definition}

\newtheorem{example}[equation]{Example}

\numberwithin{equation}{section}

\DeclareMathOperator{\im}{im}

\newcommand{\g}{\mathfrak{g}}
\renewcommand{\b}{\mathfrak{b}}
\newcommand{\uu}{\mathfrak{u}}
\renewcommand{\t}{\mathfrak{t}}
\newcommand{\B}{\mathcal{B}}
\newcommand{\Gt}{\widetilde{G}}
\newcommand{\gt}{\widetilde{\g}}
\newcommand{\dg}{\mathbb{D}_G}

\newcommand{\nubar}{\overline{\nu}}
\DeclareMathOperator{\Ad}{Ad}

\DeclareMathOperator{\spec}{Spec}

\makeatletter
\def\l@subsection{\@tocline{2}{0pt}{3pc}{5pc}{}}
\makeatother

\title{A quasi-Poisson structure on the multiplicative Grothendieck--Springer resolution}
\author{Ana B\u{a}libanu}
\address{Department of Mathematics, Harvard University, 1 Oxford Street, Cambridge, MA  02138, USA}
\email{ana@math.harvard.edu}
\date{}

\begin{document}
\maketitle

\begin{abstract}
In this note we show that the multiplicative Grothendieck--Springer space has a natural quasi-Poisson structure. The associated group-valued moment map is the resolution morphism, and the quasi-Hamiltonian leaves are the connected components of the preimages of Steinberg fibers. This is a multiplicative analogue of the standard Poisson structure on the additive Grothendieck--Springer resolution, and an explicit illustration of a more general procedure of reduction along Dirac realizations which was developed in recent work of the author and Mayrand.
\end{abstract}

\tableofcontents
\section*{Introduction}
The multiplicative Grothendieck--Springer resolution of a complex reductive group $G$ is the incidence variety 
\[\Gt=\{(g,B)\in G\times\mathcal{B}\mid g\in B\},\]
where $\B$ is the flag variety of $G$, viewed as the space of all Borel subgroups or, equivalently, of all Borel subalgebras. The first projection  
\[\mu:\Gt\longrightarrow G\]
is a generically finite map which restricts to a resolution of singularities along each Steinberg fiber, and its restriction to the unipotent cone was originally constructed by Springer \cite{spr:69}. 

The group $G$ is an example of a quasi-Poisson manifold. Such manifolds are generalizations of Poisson structures in which the Jacobi identity is twisted by a canonical trivector field, and they were introduced by Alekseev, Kosmann-Schwarzbach, and Meinrenken \cite{ale.kos.mei:02}. They are equipped with group-valued moment maps and are foliated by nondegenerate leaves, each of which carries a quasi-Hamiltonian $2$-form. In particular, the quasi-Hamiltonian leaves of $G$ are the conjugacy classes, which vary in dimension. Therefore, the quasi-Poisson structure on $G$ is singular in the sense of foliation theory. 

In this note we show that the Grothendieck--Springer resolution $\Gt$ also has a natural quasi-Poisson structure, whose group-valued moment map is $\mu$. The leaves of this structure are the connected components of the preimages of Steinberg fibers under $\mu$, and they form a regular foliation of $\Gt$. The map $\mu$ is therefore a desingularization in two ways---it resolves each singular Steinberg fiber to a smooth variety, and it resolves the singular quasi-Hamiltonian foliation of $G$ to a regular quasi-Hamiltonian foliation of $\Gt$.

%
%
%
%
%
%
%
\subsection*{The additive Grothendieck--Springer resolution}
These observations are a multiplicative analogue of the Grothendieck--Springer resolution of the Lie algebra $\g$ of $G$, which is the variety of pairs
\[\gt=\{(x,\b)\in\g\times\B\mid x\in\b\}\]
consisting of an element in the Lie algebra and a Borel subalgebra that contains it. We briefly recall its geometry, and we refer to \cite{chr.gin:97} for more details. Fixing a Borel subgroup $B$ of $G$ and writing $\b$ for its Lie algebra, there is a natural isomorphism
\begin{align*}
G\times_B\b&\xlongrightarrow{\sim}\gt\\
	[g:x]&\longmapsto (\Ad_g(x),\Ad_g(\b))
	\end{align*}
which makes $\gt$ into a vector bundle over $G/B$.

Let $\t$ be a fixed maximal Cartan in $\b$ with corresponding subgroup $T\subset B$, and let $W$ be its Weyl group. Decompose $\b=\t\oplus\uu$, where $\uu=[\b,\b]$ is the nilpotent radical of $\b$. The resolution $\gt$ fits into the diagram
\begin{equation*}
\begin{tikzcd}[row sep=large]
&\gt\arrow[ld,swap,"\mu"]\arrow[rd, "\lambda"]&\\
\g\arrow[rd,swap,"\kappa"]&&\t\arrow[ld]\\
&\t/W,&
\end{tikzcd}
\end{equation*}
where $\mu$ is the first projection, $\lambda$ is the projection of the bundle $G\times_B\b$ onto the $\t$-component of the fiber $\b$, and $\kappa$ is the composition of the adjoint quotient map with the Chevalley isomorphism
\[\g\!\sslash\! G\cong \t/W.\]

For any $s\in\t$, write $\bar{s}$ for its image in $\t/W$. When $s$ is not regular, $\kappa^{-1}(\bar{s})$ is a singular algebraic variety and the map $\mu$ restricts to a resolution of singularities
\[G\times_B(s+\uu)\cong\lambda^{-1}(s)\longrightarrow\kappa^{-1}(\bar{s}).\]
In particular, in the special case when $s=0$ this map is the Springer resolution
\[G\times_B\uu\longrightarrow\mathcal{N}\]
of the nilpotent cone $\mathcal{N}$ of $\g$.

One can equip $\gt$ with a Poisson structure through Hamiltonian reduction as follows. The action of $G$ on itself by right multiplication induces Hamiltonian actions of $G$ and $B$ on the cotangent bundle $T^*_G$. The corresponding moment maps fit into the commutative diagram 
\begin{equation*}
\begin{tikzcd}[row sep=small]
T^*_G\cong\hspace{-32pt}&G\times\g\arrow[dd,swap,"\nubar"]\arrow[rd, "\nu"]&\\
&&\g\arrow[ld]\\
&\g/\uu.&
\end{tikzcd}
\end{equation*}
Here we use the left trivialization of the cotangent bundle and we identify $\g^*$ with $\g$ and $\b^*$ with $\g/\uu$ through the Killing form, so that the moment map $\nu$ is just the second projection. Since each point in the subset $\b/\uu\subset\g/\uu$ is fixed under the action of $\b$ on $\g/\uu$, the reduced space
\[\nubar^{-1}(\b/\uu)/B=\nu^{-1}(\b)/B=G\times_B\b\]
inherits a natural Poisson structure from the canonical symplectic structure on $T^*_G$. Its symplectic leaves are the individual symplectic reductions
\[\nu^{-1}(s+\uu)/B=G\times_B(s+\uu) \quad \text{for }s\in\t,\]
each of which is an affine bundle over $G/B$. Therefore the Grothendieck--Springer resolution $\gt$ is a regular Poisson variety, and each resolution of singularities
\[G\times_B(s+\uu)\longrightarrow\kappa^{-1}(s)\]
is a symplectic resolution.

%
%
%
%
%
%
%
\subsection*{The group-valued analogue}
In the quasi-Poisson setting, the role of the cotangent bundle is played by the internal fusion double $\dg=G\times G$. A priori there is no analogue of Hamiltonian reduction with respect to the action of the Borel subgroup, because the $G$-valued moment map of $\dg$ does not descend to a moment map with values in $B$---in other words, a quasi-Poisson $G$-space is not generally quasi-Poisson for the action of a subgroup of $G$.

However, by work of Bursztyn and Crainic \cite{bur.cra:05, bur.cra:09}, quasi-Poisson manifolds are an example of a more general class of structures known as Dirac manifolds. In \cite{bal.may:22}, the author and Mayrand give a general procedure for Dirac reduction along a certain class of maps known as Dirac realizations. In this note, we show explicitly how this theory specializes to the Grothendieck--Springer space $\Gt$, using only the basics of Dirac structures. We prove the following result, as Theorems \ref{main1} and \ref{main2}.

\begin{theorem*}
There is a natural quasi-Poisson structure on the multiplicative Grothendieck--Springer resolution
\[\Gt\cong G\times_BB,\]
which is inherited from the quasi-Hamiltonian structure on $\dg$ and for which the map $\mu$ is a group-valued moment map. Its quasi-Hamiltonian leaves are the twisted unipotent bundles
\[G\times_BtU\quad \text{for } t\in T.\]
\end{theorem*}

The fiber bundle $G\times_BtU$ resolves the singularities of the Steinberg fiber $F_t$, which is a possibly singular quasi-Hamiltonian variety. This resolution is quasi-Hamiltonian, in the sense that the pullback of the corresponding $2$-form along the desingularization map
\[G\times_BtU\longrightarrow F_t\]
agrees with the quasi-Hamiltonian structure on the leaf $G\times_BtU$. This quasi-Hamiltonian structure first appeared in work of Boalch \cite[Section 4]{boa:11} in the setting of moduli spaces, and the quasi-Poisson structure on $\Gt$ which induces it can also be constructed using results of of Li-Bland and \v{S}evera \cite[Theorem 5]{lib.sev:15}.

We remark that the resolution $\Gt$ also carries a Poisson structure, which was introduced by Evens and Lu \cite{eve.lu:07} and which is obtained from a Poisson structure on the double $\dg$ through coisotropic reduction. In this work, the authors equip $G$ with a compatible Poisson structure that can be viewed as the semi-classical limit of the quantum group $U_q(\g)$, and with respect to which the resolution map $\mu$ is Poisson. The $T$-orbits of the symplectic leaves of $G$ are intersections of conjugacy classes and Bruhat cells, and the singularities of their closures are resolved by the $T$-orbits of symplectic leaves in $\Gt$, which are regular Poisson manifolds.

\textbf{Acknowledgements.} This work was partially supported by a National Science Foundation MSPRF under award DMS--1902921.

%
%
%
%
%
%
%
\section{Dirac manifolds and quasi-Poisson structures}
\subsection{Dirac structures}
We begin by giving some background on Dirac manifolds, and we refer the reader to \cite{bur:13} for more details. Let $M$ be a (real or complex) manifold and let $\eta\in\Omega^3(M)$ be a closed $3$-form. A \emph{$\eta$-twisted Dirac structure} on $M$ is a subbundle $L\subset T_M\oplus T_M^*$ such that
\begin{itemize}
\item $L$ is Lagrangian with respect to the nondegenerate symmetric pairing 
\[\langle (X,\alpha),(Y,\beta)\rangle=\alpha(Y)+\beta(X),\,\,\text{and}\] 
\item the space of sections $\Gamma(L)$ is closed under the $\eta$-twisted Dorfman bracket
\[\llbracket(X,\alpha),(Y,\beta)\rrbracket=([X,Y], \mathcal{L}_X\beta+\iota_Y\alpha+\iota_{X\wedge Y}\eta).\]
\end{itemize}
While the Dorfman bracket is not a Lie bracket on $T_M\oplus T_M^*$, it restricts to a Lie bracket on the sections of the subbundle $L$. In this way $L$ becomes a Lie algebroid over $M$, with anchor map the first projection
\begin{align*}
p_T:\quad L\quad&\longrightarrow T_M\\
	(X,\alpha)&\longmapsto X.
	\end{align*}

\begin{example}
Let $\omega\in\Omega^2(M)$ be any $2$-form, and let $\eta=-d\omega$. Then the graph
\[L_\omega=\{(X,\omega^\flat(X))\mid X\in T_M\}\]
is a $\eta$-twisted Dirac structure on $M$. Such Dirac structures are called \emph{nondegenerate} or \emph{presymplectic}, and are characterized by the property 
\[L_\omega\cap T_M^*=0.\]
Note that here and throughout we view $T_M$ and $T_M^*$ as subbundles of $T_M\oplus T_M^*$ via the coordinate embeddings. In particular, a Dirac structure $L$ on $M$ is induced by a symplectic form if and only if it is non-twisted and 
\[L\cap T_M^*=L\cap T_M=0.\]
\end{example}

More generally, if $(M,L)$ is a $\eta$-twisted Dirac manifold, the image of the anchor map $p_T$ is an involutive generalized distribution which integrates to a foliation of $M$ by \emph{nondegenerate} or \emph{presymplectic leaves}. Each leaf $S$ carries a natural presymplectic form $\omega_S$, which has the property that
\[d\omega_S=-\eta_{\vert S}.\]

\begin{example}
Let $\pi\in\mathcal{X}^2(M)$ be a bivector on $M$ and suppose that there is a $3$-form $\eta\in\Omega^3(M)$ such that
\[[\pi,\pi]=2\pi^\#(\eta),\]
where the left-hand side denotes the Schouten--Nijenhuis bracket. Then $\pi$ is a \emph{twisted Poisson structure}, and its graph
\[L_\pi=\{(\pi^\#(\alpha),\alpha)\mid\alpha\in T_M^*\}\]
is a $\eta$-twisted Dirac structure on $M$. Conversely, a Dirac structure $L$ on $M$ is induced by a Poisson bivector if and only it is non-twisted and
\[L\cap T_M=0.\]
In this case, the foliation of $M$ by nondegenerate leaves is exactly the symplectic foliation.
\end{example}

Given a map $f:M\longrightarrow N$ and a $\eta_N$-twisted Dirac structure $L_N$ on $N$, the \emph{pullback} of $L_N$ under $f$ is the generalized distribution
\[f^*L_N=\{(X, f^*\alpha)\in T_M\oplus T_M^*\mid (f_*X,\alpha)\in L_N\}.\]
If $f^*L_N$ is a smooth bundle, it defines a $f^*\eta_N$-twisted Dirac structure on $M$ \cite[Proposition 1.10]{bur:13}. When $M$ is equipped with a twisted Dirac structure $L_M$, the map $f$ is called \emph{backward-Dirac} (or \emph{b-Dirac}) if 
\[L_M=f^*L_N.\]
Such maps generalize the pullbacks of differential forms. In particular, when $(M,\omega_M)$ and $(N,\omega_N)$ are nondegenerate Dirac manifolds, $f$ is b-Dirac if and only if
\[f^*\omega_N=\omega_M.\]

Similarly, if $L_M$ is a $\eta_M$-twisted Dirac structure on $M$, the \emph{pushforward} of $L_M$ under $f$ is the generalized distribution
\[f_*L_M=\{(f_*X,\alpha)\in T_N\oplus T_N^*\mid (X,f^*\alpha)\in L_M\},\]
as long as it is well-defined. 
When $N$ is equipped with a twisted Dirac structure $L_N$, the map $f$ is called \emph{forward-Dirac} (or \emph{f-Dirac}) if at every point it satisfies
\[L_N=f_*L_M.\]
Such maps generalize pushforwards of vector fields. In particular, when $(M,\pi_M)$ and $(N,\pi_N)$ are Poisson manifolds, $f$ is f-Dirac if and only if 
\[f_*\pi_M=\pi_N.\]

When $f:M\longrightarrow N$ is a diffeomorphism, it is f-Dirac if and only if it is b-Dirac, and in this case it is called a \emph{Dirac diffeomorphism}. Suppose now that a group $G$ acts on $M$ by Dirac diffeomorphisms and that $M/G$ has the structure of a manifold such that the quotient map 
\[q:M\longrightarrow M/G\]
is a smooth submersion. Write $\g$ for the Lie algebra of $G$ and 
\begin{align*}
\rho_M:\g&\longrightarrow\mathcal{X}(M)\\
\xi&\longmapsto\xi_M
\end{align*}
for the infinitesimal action map. In this case the pushforward $f_*L_M$ is a Dirac structure on $M/G$ if the action of $G$ on $M$ is \emph{regular}---that is, if the generalized distribution
\[\rho_M(\g)\cap L_M\subset T_M\]
has constant dimension \cite[Proposition 1.13]{bur:13}.

%
%
%
%
%
%
%
\subsection{Quasi-Poisson manifolds}
Let $G$ be a (real or complex) Lie group whose Lie algebra $\g$ carries an invariant, nondegenerate, symmetric bilinear form $(\cdot,\cdot)$. The \emph{Cartan $3$-form} of $G$ is the invariant $3$-form induced by the element $\eta\in\wedge^3\g^*$ defined by
\[\eta(x,y,z)=\frac{1}{12}(x,[y,z])\quad\text{for all }x,y,z\in\g,\]
and we denote by $\chi\in\wedge^3\g$ the $3$-tensor which corresponds to it under the induced identification $\g\cong\g^*$.

Through the infinitesimal action map, $\chi$ generates a canonical bi-invariant trivector field 
\[\chi_M\in\mathcal{X}^3(M).\]
A \emph{quasi-Poisson} structure on $M$ is a bivector field $\pi\in\mathcal{X}^2(M)$ whose Schouten bracket satisfies
\[[\pi,\pi]=\chi_M.\]
This notion was first introduced in a series of papers by Alekseev, Malkin, and Meinrenken \cite{ale.mal.mei:98}, by Alekseev and Kosmann-Schwarzbach \cite{ale.kos:00}, and by Alekseev, Kosmann-Schwarzbach, and Meinrenken \cite{ale.kos.mei:02}. Viewed as a skew-symmetric bracket on the space of functions on $M$, a quasi-Poisson structure is a biderivation which satisfies a $\chi_M$-twisted version of the Jacobi identity. In particular, when the action of $G$ on $M$ is trivial, the Cartan trivector field $\chi_M$ vanishes and we recover the usual definition of a Poisson manifold.

Consider the map 
\begin{align*}
\sigma:\mathfrak{g}&\longrightarrow T_G^*\\
		\xi&\longmapsto \frac{1}{2}(\xi^R+\xi^L)^\vee,
		\end{align*}
where $\xi^R$ and $\xi^L$ are the right- and left-invariant vector fields induced by the Lie algebra element $\xi$, and $v^\vee\in T^*_G$ is the $1$-form corresponding to the vector field $v\in T_G$ under the isomorphism given by left-invariant bilinear form induced by $(\cdot,\cdot)$ on $T_G$. We let 
\[\sigma^\vee:T_G^*\longrightarrow \g\] 
be its adjoint. The quasi-Poisson manifold $(M,\pi)$ is \emph{Hamiltonian} if it is equipped with a $G$-equivariant map 
\[\Phi:M\longrightarrow G,\]
called a \emph{group-valued moment map}, which satisfies the condition
\begin{equation}
\label{mocondpi}
\pi^\#\circ\Phi^*=\rho_M\circ\sigma^\vee.
\end{equation}

The Hamiltonian quasi-Poisson manifold $(M,\pi,\Phi)$ is \emph{nondegenerate} if the bundle map
\begin{align}
\label{nondeg}
T^*_M\oplus\g&\longrightarrow T_M\\
		(\alpha,\xi)&\longmapsto \pi^\#(\alpha)+\xi_M\nonumber
		\end{align}
is surjective. In this case, M carries a \emph{quasi-Hamiltonian} $2$-form $\omega\in\Omega^2(M)$ which satisfies the compatibility condition
\begin{equation}
\label{compat}
\pi^\#\circ\omega^\flat=C,
\end{equation}
where 
\[C\coloneqq1-\frac{1}{4}\rho_M\circ\rho^\vee\circ\Phi_*\]
and 
\[\rho^\vee:T_G\longrightarrow\g\]
is the adjoint of the infinitesimal action of $G$ on itself by conjugation \cite[Lemma 10.2]{ale.kos.mei:02}. This $2$-form has the property that
\[d\omega=-\Phi^*\eta,\]
and the moment map condition \eqref{mocondpi} can be rewritten as
\begin{equation}
\label{mocond}
\omega^\flat\circ\rho_M=\Phi^*\circ\sigma.
\end{equation}
Moreover, if \eqref{nondeg} fails to be surjective, its image is an integrable generalized distribution and the quasi-Poisson manifold $M$ is foliated by quasi-Hamiltonian leaves.

%
%
%
%
%
%
%
\subsection{Quasi-Poisson bivectors as Dirac structures}
The quasi-Poisson structure on the $G$-manifold $(M, \pi,\Phi)$ is encoded \cite[Theorem 3.16]{bur.cra:05} by the $\Phi^*\eta$-twisted Dirac bundle
\[L=\{(\pi^\#(\alpha)+\xi_M,C^*(\alpha)+\Phi^*\sigma(\xi))\mid \alpha\in T^*_M, \,\xi\in\g\}.\]
The image of its projection onto the tangent component is the generalized distribution given by the map \eqref{nondeg}, and the nondegenerate leaves associated to this Dirac structure are precisely the quasi-Hamiltonian leaves. In particular, if $M$ is quasi-Hamiltonian with $2$-form $\omega\in\Omega^2(M)$, by \eqref{compat} and \eqref{mocond} this bundle can be writen
\[L=\{(X,\omega^\flat(X))\mid X\in T_M\}.\]

\begin{example}
\cite[Proposition 3.1]{ale.kos.mei:02} The group $G$ has a natural quasi-Poisson structure relative to the conjugation action, called the \emph{Cartan--Dirac structure}, whose moment map is the identity. The associated Dirac bundle \cite[Example 3.4]{bur.cra:05} is
\begin{equation}
\label{cd}
L_G=\left\{(\rho(\xi),\sigma(\xi))\mid \xi\in\g\right\}=\left\{(\xi^L-\xi^R,\sigma(\xi))\mid \xi\in\g\right\}.
	\end{equation}
Projecting onto the tangent component, we see that the nondegenerate leaves of this structure are the conjugacy classes. Therefore the Cartan--Dirac structure on $G$ can be seen as a multiplicative analogue of the classical Kirillov--Kostant--Souriau Poisson structure on $\g^*\cong\g$.
\end{example}

\begin{example}
\cite[Example 5.3]{ale.kos.mei:02} The \emph{internal fusion double} $\dg\coloneqq G\times G$ has a natural nondegenerate quasi-Poisson structure relative to the $G\times G$-action 
\[(g_1,g_2)(a,b)=(g_1ag_2^{-1},g_2bg_2^{-1}).\]
This structure is a multiplicative counterpart of the canonical symplectic structure on the cotangent bundle of $G$, viewed under the identification
\[T^*_G\cong G\times\g\]
induced by left-trivialization and by the invariant inner product. Its moment map is given by
\begin{align}
\label{phi}
\Phi:\,\,\,\dg\,\,\,&\longmapsto G\times G \\
(a,b)&\longmapsto(aba^{-1}, b^{-1}).\nonumber
\end{align}
\end{example}

Viewing the quasi-Poisson $G$-manifold $(M,\pi,\Phi)$ as a Dirac manifold with $\Phi^*\eta$-twisted Dirac structure $L$, the moment map $\Phi$ is a f-Dirac map which satisfies the additional nondegeneracy condition
\begin{equation}
\label{cond}
\ker\Phi_*\cap L=0.
\end{equation}
In fact, this property completely characterizes Hamiltonian quasi-Poisson manifolds \cite[Proposition 3.20]{bur.cra:05}---suppose that $(M,L)$ is a Dirac manifold and that
\[\Phi:(M,L)\longrightarrow (G,L_G)\]
is a f-Dirac map which satisfies \eqref{cond}. Then, for any $(\rho(\xi),\sigma(\xi))\in L_G$, there is a unique pair $(X_\xi,\alpha_\xi)\in L$ such that 
\[\Phi_*X_\xi=\rho(\xi)\qquad\text{and}\qquad \Phi^*\sigma(\xi)=\alpha_\xi.\]
This gives an infinitesimal action of the Lie algebra $\g$ on $M$ via
\begin{align}
\label{induced}
\rho_M:\g&\longrightarrow\mathcal{X}(M)\\
\xi&\longmapsto X_\xi.\nonumber
\end{align}
Moreover, for any $\alpha\in T^*M$, there is a unique $X_\alpha\in T_M$ such that 
\[\Phi_*X_\alpha=\sigma^{\vee*}\rho_M^*(\alpha)\qquad\text{and}\qquad (X_\alpha,C^*(\alpha))\in L.\]
The map
\begin{align*}
\pi^\#: T^*_M&\longrightarrow T_M\\
		\alpha&\longmapsto X_\alpha
		\end{align*}
defines a quasi-Poisson bivector $\pi\in\mathcal{X}^2(M)$ relative to the action \eqref{induced}, whose associated moment map is $\Phi$.\\

%
%
%
%
%
%
%
\section{The multiplicative Grothendieck--Springer resolution}
\subsection{The Grothendieck--Springer space}
From now on let $G$ be a reductive complex group, so that the nondegenerate bilinear form on every simple factor of $\g$ is given by a scalar multiple of the Killing form, and once again write $\B$ for the flag variety of all Borel subgroups. The \emph{Grothendieck--Springer simultaneous resolution} of $G$ is the variety of pairs
\[\Gt\coloneqq\{(g',B')\in G\times\B\mid g'\in B'\}.\]
There is a natural map 
\begin{equation}
\label{mu}
\mu:\Gt\longrightarrow G
\end{equation}
given by the first projection, whose general fiber is finite---when $g'\in G$ is regular and semisimple, the points of $\mu^{-1}(g')$ are permuted freely and transitively by $W$, the Weyl group of $G$. 

Fixing a Borel subgroup $B$ of $G$, the isomorphism
\begin{align*}
G\times_BB&\xlongrightarrow{\sim}\Gt \\
	[g:b]&\longmapsto (gbg^{-1}, gBg^{-1})
	\end{align*}
realizes $\Gt$ as a fiber bundle over $G/B\cong\B$. Let $T$ be a maximal torus of $B$ and let $U=[B,B]$ be its unipotent radical, so that we have a splitting $B=TU$. Since $B$ stabilizes each coset $tU$, there is a well-defined map
\begin{align*}
\lambda: G\times_BB&\longrightarrow T \\
				[g:tu]&\longmapsto t
				\end{align*}
whose fibers
\[\lambda^{-1}(t)=G\times_B tU\]
are multiplicative analogues of twisted cotangent bundles over $G/B$.

Let $G\!\sslash\! G=\spec\mathbb{C}[G]^G$ be the adjoint quotient of $G$. By the Chevalley theorem, the restriction map gives an isomorphism of algebras
\[\mathbb{C}[G]^G\cong\mathbb{C}[T]^W\]
and therefore an identification of varieties
\[G\!\sslash\! G\cong T/W.\]
The two quotient maps 
\[G\longrightarrow G\!\sslash\! G\qquad\text{ and }\qquad T\longrightarrow T/W\]
fit into a commutative diagram
\begin{equation}
\label{gs}
\begin{tikzcd}[row sep=large]
&\Gt\arrow[ld,swap,"\mu"]\arrow[rd, "\lambda"]&\\
G\arrow[rd,swap,"\kappa"]&&T\arrow[ld]\\
&T/W&
\end{tikzcd}
\end{equation}
whose restriction to the regular locus of $G$ is Cartesian.

For any $t\in T$, let $\bar{t}$ be its image in $T/W$. The fibers of $\kappa$, which are irreducible subvarieties of $G$ of codimension equal to the rank, are called \emph{Steinberg fibers}, and each is a union of finitely many conjugacy classes. In particular, the Steinberg fiber
\[F_{t}\coloneqq\kappa^{-1}(\bar{t})\]
contains a unique regular conjugacy class, which is open and dense, and the unique semisimple conjugacy class $G\cdot t$, which is closed and of minimal dimension \cite[Theorem 6.11 and Remark 6.15]{ste:65}. In particular, if $t\in T$ is regular, then the fiber $F_{t}$ is simply the orbit $G\cdot t$ and therefore it is a smooth variety isomorphic to the quotient $G/T.$ 

When $t\in T$ is not regular, the Steinberg fiber $F_{t}$ is generally singular. Diagram \eqref{gs} gives a natural surjection $\lambda^{-1}(t)\longrightarrow\kappa^{-1}(\bar{t})$, and we obtain a proper birational map
\[G\times_BtU\longrightarrow F_t\]
which is a resolution of singularities. In the special case of the identity element $1\in T$, the Steinberg fiber $\mathcal{U}\coloneqq F_1$ is the variety of all unipotent elements of $G$, and this birational map is precisely the Springer resolution
\[G\times_BU\longrightarrow\mathcal{U}.\]
%

%
%
%
%
%
%
%
\subsection{A quasi-Poisson structure on $\Gt$}
In this section we will show that $G\times_BB$ carries a natural quasi-Poisson structure for the left action of $G$ given by
\begin{equation}
\label{residual}
h\cdot[g:b]=[hg:b].
\end{equation}
Consider the diagram
\begin{equation}
\label{bigdiag}
\begin{tikzcd}[row sep=large]
		&G\times B\arrow[ld, swap, "q"]\arrow[rr, hook, "\jmath"] \arrow[rd, "\Phi"]		&	&\dg\arrow[rd, "\Phi"] &\\
G\times_BB\arrow[rd, swap, "\mu"] 	&&G\times B\arrow[ld, "p"]\arrow[rr, hook, "\imath"] 		&& G\times G \\
							&G, &&&			
\end{tikzcd}
\end{equation}
where $p$ and $q$ are the natural quotient maps, $\mu$ is the map defined in \eqref{mu}, and $\Phi$ is the quasi-Poisson moment map of the internal fusion double given by \eqref{phi}.

The bundle $\jmath^*L_{\dg}$ is the graph of the two-form $\jmath^*\omega$, and is therefore a Dirac structure on $G\times B$. To check that it descends to the desired Dirac structure on the Grothendieck--Springer resolution, we will begin with the following simple lemma.

\begin{lemma}
\label{kernel}
Let $\xi\in\mathfrak{b}$. Then $T_B\subset\ker\sigma(\xi)$ if and only if $\xi\in\mathfrak{u}$.
\end{lemma}
\begin{proof}
Fix a point $tu\in B$, with semisimple part $t\in T$ and unipotent part $u\in U$, and write $\xi=s+n\in\mathfrak{t}\oplus\mathfrak{u}$. Any vector in $T_{tu}B$ is of the form $x^R$ for some $x\in\mathfrak{b}$, and we have
\begin{align*}
x^R\in\ker(\xi^R+\xi^L)^\vee\quad\text{for all }x\in\mathfrak{b}\quad&\iff\quad (\xi+\Ad_{tu}\xi,x)=0\quad\text{for all }x\in\mathfrak{b}\\
												&\iff\quad \xi+\Ad_{tu}\xi\in\mathfrak{u}\\
												&\iff\quad s=0,
												\end{align*}
where the second equivalence follows from the fact that $\mathfrak{b}^\perp=\mathfrak{u}$ under the Killing form.
\end{proof}

\begin{proposition}
\label{regact}
The action of $B$ on $G\times B$ defined by
\[h\cdot(g,b)=(gh^{-1},hbh^{-1})\]
for $h\in B$ and $(g,b)\in G\times B$ is a regular Dirac action with respect to the Dirac structure $\jmath^*L_{\dg}$.
\end{proposition}
\begin{proof}
Since the second copy of $G$ acts on $\dg$ by Dirac automorphisms, the same is true for the action of $B$ on the submanifold $G\times B$. It remains only to show that this action is regular. 

We have
\begin{align*}
\rho_{\dg}(0\oplus\mathfrak{b})\cap \jmath^*L_{\dg}&=\{\rho_{\dg}(0,\xi)\mid \xi\in\mathfrak{b}\text{ and }T_{G\times B}\subset\ker\Phi^*\sigma(0,\xi)\}\\
			&=\{\rho_{\dg}(0,\xi)\mid \xi\in\mathfrak{b}\text{ and }T_{G\times B}\subset\ker\sigma(0,\xi)\}\\
			&=\{\rho_{\dg}(0,\xi)\mid \xi\in\mathfrak{u}\},
\end{align*}
where the last equality follows from Lemma \ref{kernel}. Since the action of $B$ on $G\times B$ has trivial stabilizers, this implies that the intersection $\rho_{\dg}(0\oplus\mathfrak{b})\cap \jmath^*L_{\dg}$ has constant rank and therefore that this action is regular.
\end{proof}

Because the Borel subalgebra satisfies %
\[\b^\perp=[\b,\b],\]
the restriction of the Cartan $3$-form $\eta$ to the subgroup $B$ vanishes. Chasing diagram \eqref{bigdiag}, we see that
\begin{align*}
q^*\mu^*\eta&=\Phi^*p^*\eta \\
			&=\Phi^*\imath^*(\eta_1,\eta_2) \\
			&=\jmath^*\Phi^*(\eta_1,\eta_2),
			\end{align*}
so the twist of the Dirac structure $\jmath^*L_{\dg}$ is in the image of $q^*$. By \cite[Proposition 1.13]{bur:13}, Proposition \ref{regact} then implies that the pushforward 
\[L_{\Gt}\coloneqq q_*\jmath^*L_{\dg}\]
is a smooth vector bundle and therefore a $\mu^*\eta$-twisted Dirac structure on $\Gt$. We now show that it corresponds to a quasi-Poisson bivector.

\begin{theorem}
\label{main1}
The Dirac structure $L_{\Gt}$ is induced by a quasi-Poisson structure on $G\times_BB$ whose moment map is $\mu$. 
\end{theorem}

\begin{proof}
To show that $L_{\Gt}$ defines a quasi-Poisson structure with moment map $\mu$, we must show (i) that $\mu$ is a f-Dirac map which satisfies the nondegeneracy condition \eqref{cond}, and (ii) that the induced $G$-action on $G\times_BB$ coincides with the action given in \eqref{residual}. 

(i) To see that $\mu$ is f-Dirac, we use diagram \eqref{bigdiag} to compute
\begin{align*}
L_G&=p_{*}\imath^*\Phi_*L_{\dg} \\
	&=p_{*}\Phi_*\jmath^*L_{\dg}\\
	&=\mu_*q_*\jmath^*L_{\dg}=\mu_*L_{\Gt},
	\end{align*}
where the second equality follows from \cite[Lemma 1.19]{bal:22}. Checking the nondegeneracy condition
\[\ker\mu_*\cap L_{\Gt}=0,\]
is equivalent to showing that
\begin{equation}
\label{goal}
\ker(p_{*}\Phi_*)\cap \jmath^*L_{\dg}\subset \rho_{\dg}(0\oplus\mathfrak{b}).
\end{equation}
Suppose that $(X,0)\in \jmath^*L_{\dg}$ satisfies $p_{*}\Phi_*(X)=0$. Then there is a $1$-form $\alpha\in T_{G\times B}^\circ$ such that
\[(X,\alpha)\in L_{\dg}.\]
Moreover, since $T_{G\times B}^\circ=\Phi^*T_{G\times B}^\circ$ and since $\Phi$ is f-Dirac, there is a further $\beta\in T_{G\times B}^\circ$ such that $(\Phi_*X,\beta)$ is an element of the Cartan--Dirac structure $L_{G\times G}$. Since $p_{*}\Phi_*X=0$ and since $X$ is tangent to $G\times B$, we obtain
\[\Phi_*X=\rho(0,\xi)\]
for some $\xi\in\mathfrak{b}$. By \eqref{cd}, this means that $\beta=\sigma(0,\xi)$ and we get
\[(X,\Phi^*\beta)\in L_{\dg} \qquad\text{and}\qquad (\rho_{\dg}(0,\xi),\Phi^*\beta)\in L_{\dg}.\]
Since $\Phi_*X=\Phi_*\rho_{\dg}(0,\xi),$ the nondegeneracy condition \eqref{nondeg} applied to the moment map $\Phi$ implies that $X=\rho_{\dg}(0,\xi),$ proving \eqref{goal}.

It remains to show that the infinitesimal action of $\g$ on $G\times_BB$ induced by the map $\mu$ agrees with the action defined in \eqref{residual}. For this, let $\xi\in\g$ and notice that, since $L_{\dg}$ is quasi-Poisson for the action of $G\times G$ on $\dg$, we have
\[(\rho_{\dg}(\xi,0),\Phi^*\sigma(\xi,0))\in L_{\dg}.\]
Since $\rho_{\dg}(\xi,0)$ is tangent to $G\times B$, it follows that
\[(\rho_{\dg}(\xi,0),\jmath^*\Phi^*\sigma(\xi,0))\in L_{G\times B}.\]
Moreover, by chasing diagram \eqref{bigdiag} we see that
\[\jmath^*\Phi^*\sigma(\xi,0)=\Phi^*\imath^*\sigma(\xi,0)=\Phi^*p^*\sigma(\xi)=q^*\mu^*\sigma(\xi),\]
and we get
\[(q_*\rho_{\dg}(\xi,0),\mu^*\sigma(\xi))\in L_{\Gt}.\]
Since $\mu_*q_*\rho_{\dg}(0,\xi)=\rho(0,\xi)$, by \eqref{induced} this implies that $\xi$ acts on $G\times_BB$ by the vector field $q_*\rho_{\dg}(\xi,0)$, which is precisely the vector field corresponding to $\xi$ under the action \eqref{residual}.
\end{proof}

\begin{theorem}
\label{main2}
The quasi-Hamiltonian leaves of $L_{\Gt}$ are the twisted unipotent bundles 
\[G\times_BtU\quad \text{for }t\in T.\]
\end{theorem}
\begin{proof}
Let $X\in T_{G\times B}$. The pushforward $q_*X$ is contained in $p_T(L_{\Gt})$ if and only if
\begin{align}
\label{eq1}
\jmath^*\omega^\flat(X)\in\im q^*\quad&\iff\quad \omega(X,\rho_{\dg}(0,\xi))=0 \quad\text{for all }\xi\in\mathfrak{b}\\
						&\iff\quad X\in\ker\Phi^*\sigma(0,\xi) \quad\text{for all }\xi\in\mathfrak{b}\nonumber\\
						&\iff\quad \Phi_*X\in\ker\sigma(0,\xi) \quad\text{for all }\xi\in\mathfrak{b},\nonumber
						\end{align}
where the second equivalence follows from the moment map condition \eqref{mocond}.

Keeping the notation of Lemma \ref{kernel} and letting $x\in\b$, we have
\begin{align}
\label{eq2}
x^R\in\ker(\xi^R+\xi^L)^\vee\text{ for all }\xi\in\mathfrak{b} &\iff (\xi+\Ad_{tu}\xi,x)=0\text{ for all }\xi\in\mathfrak{b}\\
												&\iff (\xi,x+\Ad_{tu}^{-1}x)=0\text{ for all }\xi\in\mathfrak{b}\nonumber\\
												&\iff x\in\mathfrak{u},\nonumber
												\end{align}
where the last equivalence follows from Lemma \ref{kernel}. 

Using \eqref{eq1} and \eqref{eq2},
\[p_T(L_{\Gt})=\{q_*X\mid \jmath^*\omega^\flat(X)\in\im q^*\}=\{q_*X\mid X\in T_{G\times tU}\},\]
and so the leaves integrating the generalized distribution $p_T(L_{\Gt})$ are precisely the submanifolds of $G\times_BB$ of the form
\[G\times_BtU\quad\text{for } t\in T.\qedhere\]\\
\end{proof}

%
%
%
%
%
%
%
\bibliographystyle{plain}
\bibliography{biblio}

\end{document}